\documentclass[a4paper,12pt]{amsart}%{article}
\usepackage[T1]{fontenc}
\usepackage{lmodern, amsfonts,amsmath,amstext,amsbsy,amssymb,
amsopn,amsthm,upref,mathtools,url} \mathtoolsset{mathic=true}
\usepackage{mathrsfs}
\usepackage[usenames,dvipsnames,svgnames,table]{xcolor}
\definecolor{halfgray}
{gray}{0.55}%chapter numbers will be semi
\definecolor{webgreen}
{rgb}{0,0.4,0}
\definecolor{webbrown}
{rgb}{.8,0.1,0.1}
\definecolor{red}
{rgb}{1,0,0}
\usepackage{hyperref} \hypersetup{%
citecolor=NavyBlue, % color of links to bibliography
colorlinks=true, % false: boxed links; true: colored links
filecolor=magenta, % color of file links
linkcolor=Maroon, % color of internal links
pdfcreator={pdfLaTeX}, % creator of the document
pdffitwindow=false, % window fit to page when opened
pdfmenubar=true, % show Acrobat's menu?
pdfnewwindow=true, % links in new window
pdfproducer={Producer}, % producer of the document
pdfstartview={FitH}, % fits the width of the page to the window
pdfsubject={Mathematics}, % subject of the document
pdftoolbar=true, % show Acrobat's toolbar?
unicode=true, % non-Latin characters in Acrobat's bookmarks
urlcolor=cyan, % color of external links,
anchorcolor=cyan % color for anchor text
}
\usepackage{microtype}
\usepackage{tikz}

\newcommand \R {{ \mathbb R}}
\def\C{{\mathbb C}}

\newcommand \Z {{ \mathbb Z}}
\newcommand \N {{ \mathbb N}}

\newcommand \D {{ \,\mathrm d}}

\newcommand*{\diff}{\mathop{}\!\mathrm{d}}
\newcommand{\one}{{\rm 1\mskip-4mu l}}
\newcommand{\norm}[1]{\left\lVert#1\right\rVert}
\newcommand{\abs}[1]{\left\lvert#1\right\rvert}

\renewcommand{\sl}{%
\operatorname{{\mathfrak s \mathfrak l}}
}
\newcommand{\PSL}{%
\operatorname{PSL}
}

\newcommand{\PSO}{%
\operatorname{PSO}
}
\renewcommand{\Mc}{%
M_0}
\newcommand{\Sc}{%
S_0}
\newcommand{\Gammac}{%
\Gamma_0}
\newcommand{\Gal}{%
\mathscr{G}}
\newcommand{\Galhat}{%
\widehat{\mathscr{G}}}
\newcommand{\harm}{%
\mathcal{H}}
\newcommand{\fund}{%
\mathcal{F}}
\newcommand{\Mchio}{%
M_{\chi_{\omega}}}
\newcommand{\Mcchio}{%
M_{\chi \chi_{\omega}}}
\newcommand{\rapid}{%
\mathscr{R}}
\newcommand{\uot}{%
\mathbf{u}_{\omega}^t}
\newcommand{\ubt}{%
\mathbf{u}^t}
\newcommand{\ub}{%
\mathbf{u}}

\DeclareMathOperator{%
\vol}{vol}
\DeclareMathOperator{\Id}{Id}

\DeclareMathOperator{\Spec}{Spec}
\DeclareMathOperator{\Ad}{Ad}
\DeclareMathOperator{\Lie}{Lie}

\newtheorem{bigthm}{Theorem} % Numbered separately, as A, B, etc.
\newtheorem{theorem}{Theorem}[section]
\newtheorem{corollary}[theorem]{Corollary}
\newtheorem{lemma}[theorem]{Lemma}
\newtheorem{proposition}[theorem]{Proposition}
\newtheoremstyle{style2}%
{7pt}{9pt}{\normalfont}%
{0pt}{\bf}{.\enspace}{0pt}%
{}
\theoremstyle{style2}

\newtheoremstyle{style3}%
{7pt}{9pt}{\normalfont}%
{0pt}{\it}{.\enspace}{0pt}%
{} %
\theoremstyle{style3}

\newtheorem{remark}[theorem]{Remark\/}

\date{\today}

\author{Livio Flaminio}

\address{Univ.  Lille, CNRS, UMR 8524 - Laboratoire Paul Painlevé,
F-59000 Lille, France}

\email
  {livio.flaminio@univ-lille.fr}

  \author{Davide Ravotti}

  \address{Universit\"{a}t Wien, Department of Mathematics \\
  Oskar-Morgenstern-Platz 1, 1090 Wien, Austria}

  \email
    {davide.ravotti@gmail.com}

    \title[Abelian covers] {Abelian covers of hyperbolic surfaces:
    equidistribution of spectra and infinite volume mixing asymptotics for
    horocycle flows}

\begin{document}
    \maketitle

    \begin{abstract}
      We consider Abelian covers of compact hyperbolic surfaces.  We
      establish an asymptotic expansion of the correlations for the
      horocycle flow on $ \Z^d $-covers, thus proving a strong form of
      Krickeberg mixing. 
      We also prove that the spectral measures around $
      0 $ of the Casimir operators on any increasing sequence of finite
      Abelian covers converge weakly to an absolutely continuous
      measure.
    \end{abstract}

\section{Introduction}%
\label{sec:intro}

Let $ \mathbb{H} $ denote the upper-half plane model of the hyperbolic
plane. Any closed Riemann surface $ S $ with a hyperbolic structure can
be realized as a quotient $ \Gamma \backslash \mathbb{H} $ of $ \mathbb{H}
$ by a discrete group of isometries~$ \Gamma $.  In turn,~$ \Gamma $ can
be identified with the fundamental group of~$ S $. The horocycle flow $
\{h_t\}_{t \in \R} $ on the unit tangent bundle $ T^1S $ of~$ S $ is the
unit speed translation along the stable leaves of the geodesic flow.
Identifying $ \Gamma $ with a co-compact discrete subgroup of $ G=\PSL_2
(\R) $ acting on $ \mathbb{H} $ by M\"{o}bius transformations, and $ T^1S
$ with $ M= \Gamma \backslash G $, the flow $ \{h_t\}_{t \in \R} $ is
the action by right-multiplication on $ M $ by the upper triangular
unipotent subgroup of $ G $.

Qualitative and quantitative properties of the horocycle flow on compact
(or finite volume) quotients $ M $ have been studied at length and are
now well-understood. When the underlying surface has infinite genus,
however, much less is known. We will be interested in the case of
Abelian covers of compact surfaces $ S $. In this setting, the only
invariant Radon measures have infinite total mass
\cite{Rat2}. Babillot and Ledrappier constructed a family of mutually
singular ergodic Radon invariant measures
\cite{BaLe}, which are are quasi-invariant under the geodesic flow and
arise from positive eigenfunctions of the Laplacian via the associated
conformal measures on the boundary of the hyperbolic plane
\cite{Bab}. By a result of Sarig
\cite{Sar}, which applies to a wider class of hyperbolic surfaces, these
are the only ergodic Radon invariant measures.
%There exist infinitely many ergodic, mutually singular Radon invariant measures:
%they were completely classified by Sarig \cite{Sar}, who showed that they
%all arise from positive eigenfunctions of the Laplacian.
Ledrappier and Sarig showed that, of these infinitely many ergodic
measures, only the volume (Haar) measure $
\vol $ satisfies a so-called \emph{generalized law of large numbers}:  a
certain Cesaro summation with weights of appropriately rescaled ergodic
integrals of any $ L^1 $ observable converges almost everywhere to a
multiple of its expected value
\cite{LeSa}.

In this paper, we are interested in the mixing properties of the
horocycle flow on Abelian covers $ M_0 $ of $ M $. The generalization of
the definition of mixing to the infinite volume setting is not
straightforward and has a long history (see, e.g.,
\cite[\S1]{Len}). One standard notion of infinite volume mixing is the
so-called \emph{Krickeberg mixing} (or \emph{local mixing}). In our
setting, proving Krickeberg mixing for $ \{h_t\}_{t \in \R} $ on $ M_0 $
means that we can find a scaling rate $ \rho(t) $ such that
\[
  \lim_{t \to \infty} \rho(t) \int_{M_0} (v \circ h_t) \cdot \overline{w}
  \, \diff
  \vol =
  \vol(v)
  \vol(\overline{w}),
\] for all smooth functions $ v,w \in \mathscr{C}^{\infty}_c(M_0) $ with
compact support. In probabilistic language, it corresponds to a local
limit theorem. Krickeberg mixing and the related local limit theorems
have been investigated for several (non-uniformly) hyperbolic systems,
see, e.g.,
\cite{AaDe, DoNa, Gou, GuHa, MeTe, SzVa, Ter} and references therein.
For the geodesic flow $ \{g_t\}_{t \in \R} $ on a $ \Z^d $-cover $ M_0 $
of $ M $, Oh and Pan
\cite{OhPan} showed that
\begin{equation}
  \label{eq:mix_geo} \lim_{t \to \infty} t^{d/2} \int_{M_0} (v \circ g_t)
  \cdot \overline{w} \, \diff
  \vol = \frac{1}{(2\pi \sigma)^{d/2}}
  \vol(v)
  \vol(\overline{w}),
\end{equation}
for all continuous, compactly supported $ v,w \in \mathscr{C}_c(M_0) $,
where $ \sigma >0 $ is a constant depending on $ M_0 $ only. Later, this
result was extended by Pan
\cite{Pan} to the case where $ M $ is a non-compact manifold of finite
volume.

In a recent work, Dolgopyat, Nandori and P\`{e}ne
\cite{DNP} strengthened the result by Oh and Pan by providing a full
asymptotic expansion of the left hand side in~\eqref{eq:mix_geo}, under
some regularity assumptions on the observables $ u $ and $ v $. All the
works
\cite{OhPan, Pan, DNP} rely on the symbolic representation of the
geodesic flow and the spectral theory of the associated transfer
operator.

For the horocycle flow, Krickeberg mixing
can be derived from \cite[Theorem 4.1]{DNP}, using the $ KAK $ decomposition of
$ \PSL_2(\R) $.  In our first main result, Theorem~\ref{thm:main1} below, we
achieve the same result by a direct method.  (Indeed, it is also possible to
provide an alternative proof of \cite[Theorem 4.1]{DNP} by adapting our approach
to the case of the geodesic flow.)\enspace One advantage of our method is that
it characterizes the constant $\sigma$ in~\eqref{eq:mix_geo} in geometric terms.
Furthermore, our direct approach clarifies that the leading terms of the
correlation function arise from the ``merging'' of the discrete series of the
compact surface into the complementary series weakly appearing in the Abelian
cover.  The subject of further investigation is to generalize our techniques to
obtain precise mixing asymptotics in the case of Abelian covers of finite
volume, non-compact surfaces.

Our methods use representation theory of $ G $. The crucial difference
with the compact (or finite volume) case is the absence of a spectral
gap. In fact, the analysis of the asymptotics of the correlations relies
on the study of the irreducible representations close to the trivial
representation, and hence on the spectrum of the Casimir operator close
to $ 0 $. 
For a sequence of increasing finite Abelian covers $ S_k $ of a given compact 
surface $ S $, the distribution of the resonances of the
Laplace-Beltrami operators on $ S_k $ in a neighbourhood of $ 0 $ has 
been studied by Jakobson, Naud and Soares in~\cite{JNS}.
They proved that the spectral measures around $ 0 $  converge weakly 
to an absolutely continuous measure.
In our second main result, Theorem~\ref{thm:main2}, we provide an alternative 
proof of their result.

%%%%%

\subsection{Abelian covers of compact hyperbolic surfaces}

We now introduce the notation we use in this paper. Let $ \Gamma $ be a
co-compact lattice in $ G=\PSL_2(\R) $, and let $ M = \Gamma \backslash
G $.  We can identify $ M $ with the unit tangent bundle of the compact
orientable hyperbolic surface $ S = \Gamma \backslash \mathbb{H} $.  We
denote by $ g \geq 2 $ its genus.

Denote by $ [\Gamma, \Gamma] $ the commutator subgroup of $ \Gamma $.
Abelian covers $ \Sc = \Gammac \backslash \mathbb{H} $ of $ S $ are in
one-to-one correspondence with intermediate subgroups $ [\Gamma, \Gamma]
\leq \Gammac \leq \Gamma $.  The quotient group $ \Gal := \Gamma /
\Gammac $ is a finitely generated Abelian group isomorphic to the
product of a free Abelian group of rank $ d $, with $ 0 \leq d \leq 2g $,
and a finite Abelian group (the \emph{torsion} of $ \Gal $).  For the
purpose of this article we may assume, by passing to a finite cover,
that $ \Gal $ has no torsion and hence $ \Gal\simeq \Z^d $.

Let $ \Mc = \Gammac\backslash G $.  Observe that, since $ \Mc \to M $ is
a normal cover, the Galois group $ \Gal $ acts on $ \Mc $ by deck
transformations; furthermore this action commutes with the action of $ G
$ by right translations on $ \Mc= \Gammac \backslash G $.  To avoid
trivialities, we shall assume henceforth that $ \Mc $ is an Abelian
cover of $ M $ with Galois group $ \Gal\simeq \Z^d $, where $ 1\le d \le
2g $.

%%%%%

\subsection{Horocycle flows} The right actions of $ G $ on $ M $ and on $
\Mc $ are generated by vector fields which we may identify with elements
of the Lie algebra $ \sl_2(\R) $ of $ G $.  We let
\[
  U =
  \begin{pmatrix}
    0 & 1 \\
    0 & 0
  \end{pmatrix}
  , \ X =
  \begin{pmatrix}
    \frac{1}{2} & 0 \\
    0 & -\frac{1}{2}
  \end{pmatrix}
  , \ \Theta =
  \begin{pmatrix}
    0 & \frac{1}{2} \\
    -\frac{1}{2} & 0
  \end{pmatrix}
  , \ Y =
  \begin{pmatrix}
    0 & \frac{1}{2} \\
    \frac{1}{2} & 0
  \end{pmatrix}
  \in \mathfrak{sl}_2(\R).
\] The horocycle flow $ \{h_t\}_{t \in \R} $ on $ \Mc $ (or on any
quotient of $ G $) is the homogeneous flow generated by $ U $, namely
the flow given by
\[
  h_t( \Gammac g) = \Gammac g \exp(tU).
\] The right actions of $ G $ on $ M $ and on $ \Mc $ preserve the
volume forms locally given by the Haar measure.  We normalise the volume
form $ \D
\vol $ so that $ \D
\vol ( X, Y, \Theta ) = 1 $.  The
volume form $ \D
\vol_0 $ on $ \Mc $ is is the pull-back of $ \D
\vol $.  As a consequence we obtain unitary representations of $ G $ on $
L^2(M):=L^2(M, \D
\vol) $ and on $ L^2(\Mc):=L^2(\Mc, \D
\vol_0) $.

\subsection{Infinite volume mixing asymptotics}

Our main result, Theorem~\ref{thm:main1}, provides an asymptotic expansion of the correlations
for the horocycle flow between smooth observables with compact support.
The first term in the expansion is usually called Krickeberg (or local)
mixing.  

\begin{bigthm}
  \label{thm:main1} There exists a constant $ \sigma(\Gammac)>0 $ such
  that the following holds.  Let $ v, w \in \mathscr{C}^\infty_c(\Mc) $.
  There exist $ (c_j)_{j \in \N} $ such that for all $ N \ge 1 $ we have
  \[
    \begin{split}
      \langle v \circ h_t, w \rangle = &\left(\frac{g-1}{2}\right)^{\frac
      {d}{2}} \sigma(\Gammac) \frac{%
      \vol(v)
      \vol(\overline{w})}{(\log t)^{\frac{d}{2}}} + \sum_{j=1}^N \frac{c_j}
      {(\log t)^{j+\frac{d}{2}}} \\
      &+ o\left( (\log t)^{-(N+\frac{d}{2})} \right).
    \end{split}
  \]
\end{bigthm}

The constant $ \sigma(\Gammac) $ in Theorem~\ref{thm:main1} is the determinant of an explicit period matrix of
harmonic one-forms, see \S~\ref{sec:spec_around_zero}. In the case where $ \Gammac = [\Gamma,
\Gamma] $, we have $ \mathscr{G} = H_1(M,\Z) = \Z^{2g} $, and $ \sigma (\Gammac)
= 1 $.

\begin{remark}
  Theorem~\ref{thm:main1} will be proved for a larger class of functions, which
  we call \emph{rapidly decaying functions}, which, roughly speaking,
  consists of smooth functions which decay, together with all their
  derivatives, faster than any polynomial at infinity.
\end{remark}

%%%%%

\subsection{Equidistribution of spectra of large covers}

Let $ \psi \colon \Gamma \to \Z^d $, with $ d \geq 1 $, be a surjective
homomorphism and denote $ \Gammac = \ker (\psi) $.

\sloppy
Let us fix $ d $ strictly increasing sequences $ \{N_1^{(k)}, \dots, N_d^
{(k)}\}_{k \ge 1} $ of natural numbers such that $ \min\{N_1^{(k)},
\dots, N_d^{(k)} \} \to \infty $ when $ k \to \infty $.  We define
\[
  \Gamma_k = \psi^{-1} \left( N_1^{(k)} \Z \times \cdots \times N_d^{(k)}
  \Z \right),
\] and let $ M_k = \Gamma_k \backslash \PSL_2(\R) $.  Clearly, $ M_k $
is a finite Abelian cover of $ M $.

Denote by $ \square_k $ the Casimir operator on $ L^2(M_k) $.  We show
that the spectral measures of~$ \square_k $ converge weakly, in a
neighborhood of zero, to an absolutely continuous measure.  A general
result of this type for the \emph{resonances} of the Laplacian was first
proved by Jakobson, Naud and Soares
\cite{JNS}; here, we give a different proof.

\begin{bigthm}
  \label{thm:main2} There exists $ \varepsilon >0 $ and a function $
  \zeta \in L^\infty([0,\varepsilon]) $ such that, for every continuous
  function $ f $ compactly supported on $ [0,\varepsilon) $, we have
  \[
    \lim_{k \to \infty} \frac{1}{| \Gamma / \Gamma_k |} \sum_{\lambda
    \in \Spec(\square_k) \cap [0, \varepsilon]} f(\lambda) = \int_0^\varepsilon
    f(x) x^{\frac{d}{2}-1} \zeta(x) \diff x.
  \]
\end{bigthm}

%%%%%

\subsection{Outline of the paper}

Mixing rates of the one-parameter subgroup $ \{\exp(t U)\} $ acting on
an irreducible unitary representation of $ \PSL_2(\R) $ have been
studied by several authors, often under the name ``decay of
correlations''.  The gist of the matter is that \emph{mixing rates are
uniform for all representations in the unitary dual $ \widehat{\PSL_2(\R)}
$ lying outside a (Fell) neighbourhood of the trivial representation}.
In Section~\ref{sec:strong_ratner} we make this statement precise and we prove it
using elementary methods inspired by
\cite{Rat1}.

Since the action of Galois group $ \Gal $ on $ M_0 $ by deck
transformations commutes with the regular representation of $ G $ on $ L^2
( M_0 ) $, these two actions are simultaneously ``diagonalisable''.
This means that we can decompose the Hilbert space $ L^2( M_0 ) $ as a
Hilbert direct integral of irreducible unitary representations $ \chi
\otimes \pi $ of $ \Gal\times G $ (here $ \chi $ is a character of the
Galois group~$ \Gal $ and $ \pi $ an irreducible unitary representations
of $ G $). Thus $ L^2( M_0 )= \int_{\widehat{\Gal}\times \widehat{G}}
\chi \otimes \pi \, \D m (\chi, \pi) $, where $ m $ is a spectral
measure on the product of the group of characters $ \widehat{\Gal} $ of $
\Gal $ and of the unitary dual $ \widehat{G} $ of $ G $.  In Section~\ref{sec:spectrum-an-abelian}, we make explicit this decomposition.  We
show that for any $ \chi \in \widehat{\Gal} $ the irreducible unitary
representations $ \pi $ belonging to the continuous part of $ \widehat{G}
$ occur discretely in the support of the spectral measure $ m $.  They
may be ordered as $ \pi_0(\chi) $, $ \pi_1(\chi) $, \dots, in increasing
values of the Casimir parameter, so that each $ \pi_j(\chi) $ depend
smoothly on $ \chi \in \widehat{\Gal} $.  Furthermore, among these
representations, only $ \pi_0(\chi) $ may be close to the trivial
representation of $ G $ and this may occur only if $ \chi $ is close to
the trivial character of $ \Gal $.

Thus correlations of the horocycle flows are finally estimated
separating the contributions coming from the representations $ \pi_j(\chi)
$ far away from the trivial representation ---for these the uniform
bounds mentioned above apply--- from the contributions coming from the
representations $ \pi_0(\chi) $ for $ \chi $ in a neighbourhood~$
\mathcal U $ of trivial character of~$ \Gal $. This last step is
accomplished in Section~\ref{sec:proof_main_thm}.

Finally, in Section~\ref{sec:proof_thm_B}, we prove Theorem~\ref{thm:main2}.

%%%%%%%%%%

\section{Asymptotics of correlations} % of holomorphic families of vectors}
\label{sec:strong_ratner}

\subsection{Ratner's Theorem}

We turn our attention to the problem of estimating the decay of matrix
coefficients.  For unitary representations of $ G=\PSL_2(\R) $, 
upper bounds of the matrix coefficients for analytic vectors have been
established by a number of authors \cite{BoWa, CEG, Mil, Tro}.
Ratner \cite{Rat1} extended these bounds assuming only that the 
vectors are H\"{o}lder in the direction of the rotation subgroup.  
We now recall her result in the case of finite Sobolev regularity.

Let $ \rho \colon G \to \mathcal{U}(H) $ be an irreducible unitary
representation of $ G $ on a complex separable Hilbert space $ H $.
Such representations are classified, up to unitary equivalence, by a
continuous complex parameter $ \nu \in \imath \R_{\geq 0} \cup (0,1]
\subset \C $ and a discrete parameter $ \{\nu = n-1:  n \in 2\Z_{>0} \}
\times \{\sigma = \pm 1\} $.  The first set accounts for the principal ($
\nu \in \imath \R_{\geq 0} $), the complementary ($ \nu \in (0,1) $)
series, and the trivial representation ($ \nu = 1 $); the second set
accounts for the holomorphic ($ \sigma = 1 $) and anti-holomorphic ($
\sigma = -1 $) discrete series.

We choose the normalization of the Casimir operator $ \square $ to be
\[
  \square = -X^2-Y^2+\Theta^2 = -U^2+2U\Theta + X
  - X^2,
\] so that it coincides with the Laplace-Beltrami operator when acting
on function pulled back from the hyperbolic plane.  The Casimir operator
is a generator of the centre of the enveloping algebra of $ \sl_2(\R) $
and it acts on the Hilbert space $ H=H_{\lambda} $ as the constant
\[
  \lambda = \lambda(\nu) := \frac{1-\nu^2}{4}.
\] We set $ \varepsilon(\lambda) =1 $ if $ \lambda = 1/4 $ and $
\varepsilon(\lambda) =0 $ otherwise.

Finally, let $ \Delta = \square - 2 \Theta^2 = -(X^2+Y^2
+\Theta^2) $ and, for any $ s >0 $, define the Sobolev space $ W^s(H) $
of order $ s $ to be the completion of the subspace $ \mathscr{C}^\infty
(H) $ of smooth vectors $ v \in H $ with respect to $ \|v\|_{W^s} := \|(\Id
+ \Delta)^{s/2}v\| $. The space $ W^s(H) $ is a Hilbert space with the
inner product $ \langle v,w \rangle_{W^s} = \langle (\Id + \Delta)^{s/2}v,
w \rangle $ inducing the norm $ \| \cdot \|_{W^s} $ as above.

\begin{theorem}[\cite{Rat1, BoWa, CEG, Mil, Tro}]
  \label{thm:Ratner_mix} There exists an explicit constant $ C \geq 1 $
  such that the following holds. Let $ \rho \colon G \to \mathcal{U}(H) $
  be a non-trivial irreducible unitary representation of $ G $ of
  Casimir parameter $ \lambda = \lambda(\nu) $, and let $ v,w \in W^3(H)
  $.  For all $ t \geq 1 $ we have
  \[
    | \langle \rho(\exp(tU)).v, w \rangle | \leq C \|v\|_{W^3} \, \|w\|_
    {W^3} \, t^{-1+\nu} (\log t)^{\varepsilon(\nu)}.
  \]
\end{theorem}

%In our case, we need to deal with a smooth family of unitary
%representations and, for small positive Casimir parameters, we need more
%precise information on the asymptotics of the matrix coefficients. Let
%us describe our assumptions.

For our purposes, for small positive Casimir parameters, we need more
precise information on the asymptotics of the matrix coefficients. 

  In the following, for every $ v \in H $, we will denote
\[
  \ubt.v :=\rho(\exp(tU)).v \,.
\]
We now focus on vectors in $ H $ that are
eigenfunctions of both the Casimir operator $ \square $ and the
element $ \Theta $. More precisely, let $\lambda \in [0,1/4) $ and 
$ n,m \in \Z $ be fixed, and let $ e_n $ and $ e_m $
be two vectors in $ W^3(H) $ such that $
\Theta e_j = \imath j e_j $ and $ \square
= \lambda e_j $ for $ j \in \{n,m\} $.
Following Ratner's method, we will prove the following result 
(the corresponding result for the geodesic flow can be found in 
\cite[Lemma 2.30]{BKS}).

\begin{theorem}
  \label{thm:strong_ratner} Under the assumptions above, there exists $ A = A_{n,m}$ 
 explicitly defined in \eqref{eq:def_anm} such that
  \[
    \begin{split}
      &\abs{ \langle \ubt.e_n, e_m \rangle -A_{n,m}
        t^{-1+\nu}} %\\ &\hskip 3.5cm
      \leq \frac{C}{\nu^2}\|e_n\|_{W^3} \cdot
    \|e_m\|_{W^3} \, t^{-1}
    \end{split}   
  \] for some absolute constant $ C $ and for all $ t \geq 1 $.
\end{theorem}

The remainder of this section is devoted to the proof of Theorem~\ref{thm:strong_ratner}.

\subsection{Reduction to an ODE}

Let us define
\[
Y_{n,m}(t) := \langle \ubt.e_n, e_m \rangle.
\] 
It is well-known that the $e_n$ are analytic vectors and thus 
$Y_{n,m}(t)$ is a real analytic function of $t$. 
The derivatives with respect to $ t $ are given by
\begin{equation}
  \label{eq:Ynm_deriv}
  \begin{split}
    Y_{n,m}'(t) &= \langle \ubt.Ue_n, e_m \rangle,
    \\
    %\quad \text{ and } \quad
    Y_{n,m}''(t)& = \langle \ubt.U^2e_n, e_m
    \rangle.
  \end{split}
\end{equation}
The key observation is that the
function $ Y_{n,m}(t) $ satisfies a second order linear ODE with
constant coefficients, as the next proposition shows.

\begin{proposition}
  \label{prop:ynm} 
The function $ y(t)=Y_{n,m}
  (t) $ satisfies the differential equation
  \begin{equation}
    \label{eq:diffeq_ynm} t^2y''+3ty'+4\lambda y = f(t)
  \end{equation}
  for a function $ f(t)=f_{n,m}(t) $, explicilty defined in \eqref{eq:def_of_f}, 
  which is smooth in $ t $ and satisfies
  \[
    \abs{f(t)} \leq \frac{C}{t} \|e_n\|_{W^3} \, \|e_m\|_
    {W^3} \quad \text{for all $t \geq 1$ }
  \] for an absolute constant $ C $, and
  \[
    f(0)=4\lambda Y_{n,m}(0), \quad f'(0)=(4-\nu^2)
    Y'_{n,m}(0).
  \]
\end{proposition}
\begin{proof}
  The last formulas follow immediately from equation~\eqref{eq:diffeq_ynm}.
  A standard computation gives us
  \[
    \begin{split}
      \Theta (\ubt.e_m) &= \ubt.[(\Ad_{\exp(tU)}(\Theta))
      \, e_m] \\
      &= \ubt.[(\Theta + tX + \frac{1}{2}t^2U)e_m].
    \end{split}
  \] Denoting $ y(t)=Y_{n,m}(t) $, we compute
  \begin{equation*}
    \begin{split}
      \imath m\, y(t) =& -\langle \ubt.e_n, \Theta e_m
      \rangle = \langle \Theta (\ubt.e_n), e_m
      \rangle \\
      =& \langle \ubt.(\Theta e_n), e_m \rangle +
      t \langle \ubt.(X e_n), e_m \rangle  + \frac{1}{2} t^2 
      \langle \ubt.(U e_n), e_m \rangle\\
      =& \imath n\,y(t) + t \langle \ubt.(X e_n, e_m
      \rangle + \frac{1}{2} t^2y'(t).
    \end{split}
  \end{equation*}
  Define $ q(t) = \langle \ubt.(X e_n), e_m
  \rangle $. We rewrite the equation above as
  \begin{equation}
    \label{eq:lemma_ynm_1} tq(t) = \imath (m-n)y(t) - \frac{1}{2} t^2y'(t),
  \end{equation}
  and, after differentiating, we get
  \begin{equation}
    \label{eq:lemma_ynm_2} tq'(t) = -q(t) + \imath (m-n)y'(t) - \frac{1}{2} 
    t^2y''(t) - ty'(t).
  \end{equation}
  Repeating for $ q(t) $ the computation we made for $ y(t) $, we obtain
  \[
    \begin{split}
      \imath m\,q(t) &= \langle \Theta [\ubt.(X e_n)],
      e_m \rangle \\
      &= \langle \ubt.[(\Theta X + tX^2 + \frac{1}{2}t^2U
      X)e_n, e_m \rangle.
    \end{split}
  \] Using the identity $ \Theta X = X\Theta
  -Y = X \Theta -U+\Theta $, the
  formula above yields
  \begin{equation}
    \label{eq:lemma_ynm_3} \imath mq(t) = \imath n \,q(t) - y'(t) +\imath
    n y(t) + \frac{1}{2}t^2q'(t) +t \langle \ubt.(X^2 e_n), e_m
    \rangle.
  \end{equation}
  From the assumption $ \square e_n = \lambda e_n $ we get
  \begin{equation}
    \label{eq:lemma_ynm_4}
    \begin{split}
      \lambda y(t) &= \langle \ubt.(\square e_n),
      e_m \rangle \\
      &= \langle \ubt.[(-U^2 +2U\Theta +X -
      X^2)e_n], e_m \rangle \\
      &= -y''(t) +2\imath n y'(t) + q(t) - \langle \ubt.(X^2 e_n
      ), e_m \rangle.
    \end{split}
  \end{equation}
Combining~\eqref{eq:lemma_ynm_4} with~\eqref{eq:lemma_ynm_3},
we obtain
  \begin{equation*}
    \begin{split}
&-ty''(t) +2\imath n ty'(t) + t q(t) -\lambda t y(t) \\
&\qquad =\imath (m-n)q(t) + y'(t) -\imath n y(t) - \frac{1}{2}t^2q'(t), 
    \end{split}
  \end{equation*}
and, substituting the expressions for $q(t)$ and $q'(t)$ 
from~\eqref{eq:lemma_ynm_1} and ~\eqref{eq:lemma_ynm_2},
we deduce 
  \begin{equation*}
    \begin{split}
&(t^3+4t)y''(t) + (3t^2+4) y'(t) + 4\lambda t y(t) \\
&\qquad = 4 \imath (m+n) t y'(t) + 2 \imath (m+n) y(t)  + 4 \frac{(m-n)^2}{t}y(t). 
    \end{split}
  \end{equation*}
We can rewrite the above equation as 
  \begin{equation}
    \label{eq:lemma_ynm_sol} \big((t^3+4t) y'\big)'(t) + 4\lambda
    t y(t) = g(t),
  \end{equation}
  where
  \begin{equation}
    \label{eq:lemma_ynm_f} g(t) = 4\imath (m+n) t y'(t) + 2 \imath (m+n)y(t)
    + 4 \frac{(m-n)^2}{t}y(t).
  \end{equation}

Note that, if $m \neq n$, then $y(0)=0$ so that $g$ is bounded for all $t$.
  Integrating~\eqref{eq:lemma_ynm_sol}, we have
  \begin{equation}
    \label{eq:Abelian_covers:5} y'(t)=\frac {1}{t^3+4t} \int_0^t h(s)\,
    \D s, \quad \text{ where } \quad h(t)= -4\lambda t \,y(t) +
    g(t).
  \end{equation}
  Observe that, since $h$ is bounded, the integral above is well defined and the expression on
  the right-hand side above converge as $ t\to 0 $.
  %In fact, if $m=n$, then $h(t)= -\imath (2m)y(0) + O(t) $. If $m\not=n$, then
  %$y(t) = t y'(0) + O(t^2)$ and therefore $h(t)= (m-n)^2 y'(0) + O(t) $ (which
  %implies $y'(0)=0 $ if $m-n \neq \pm 2$, as it should).

  We now estimate $ h(t) $. Let us first notice that
  \[
    (1+|n|+|m|+n^2+m^2) |y(t)| \leq \sum_{k+\ell \leq 2} \|e_n\|_
    {W^k} \, \|e_m\|_{W^\ell},
  \] where we used~\eqref{eq:Ynm_deriv} and the Cauchy-Schwarz
  Inequality. Similarly, since there exists an absolute constant $ C
  \geq 1 $ such that
  \[
    |n| \cdot \|U e_n\| = \|U \Theta e_n\|
    \leq C \|e_n\|_{W^2,}
  \] then, we can bound
  \[
    (1+|n|+|m|) |y'(t)| \leq C \sum_{k+\ell \leq 2} \|e_n\|_{W^k}
    \, \|e_m\|_{W^\ell}.
  \] Thus, from~\eqref{eq:lemma_ynm_f}, we deduce
  \[
    |h(t)| \leq 20 C \left( \sum_{k+\ell \leq 2} \|e_n\|_{W^k} \,
    \|e_m \|_{W^\ell} \right) \, t \quad \text{ for all } t\geq
    1.
  \] By~\eqref{eq:Abelian_covers:5}, for all $ t \geq 1 $ we have
  \[
    |y'(t)| \leq 10 C \left( \sum_{k+\ell \leq 2} \|e_n\|_{W^k}
    \, \|e_m \|_{W^\ell} \right) \, t^{-1}.
  \] Rewriting the equation~\eqref{eq:lemma_ynm_sol} as
  \[
    t^2 y'' = -4 y'' -3 t y' - 4y'/t -4 \lambda y + g(t)/t,
  \] we deduce that, for all $ t\geq 1 $,
  \[
    \begin{split}
      |y''(t)| &\leq \frac{4}{t^2}|y''(t)| + \frac{7}{t}|y'(t)| + \frac{|h
      (t)|}{t^3} \\
      &\leq 100 C \left( \sum_{k+\ell \leq 2} \|e_n\|_{W^k} \, \|e_m
      \|_{W^\ell} \right) \, t^{-2}.
    \end{split}
  \] Finally, rewriting once more the equation~\eqref{eq:lemma_ynm_sol}
  as
  \[
    t^2y''+3ty'+4\lambda y = f(t)
  \] with
  \begin{equation}\label{eq:def_of_f}
    \begin{split}
      f(t) &= -4 y'' - \frac {4} {t} y' + \frac {g(t)}{t}\\
      &= -4 y'' -\frac{4}{t}y' - 4\imath (m+n) y'(t) - 2\imath \frac{m+n}
      {t}y(t) + 4\frac{(m-n)^2}{t^2}y(t)
    \end{split}
  \end{equation}
and combining the previous estimates, we conclude the bound
  \[
    \abs{f(t)} \leq \frac{500C}{t}\left( \sum_{k+\ell \leq 3} \|e_n\|_
    {W^k} \, \|e_m\|_{W^\ell} \right),
  \] which proves the result. 
\end{proof}

We now use Proposition~\ref{prop:ynm} to estimate $ Y_{n,m}(t) $ and prove Theorem~\ref{thm:strong_ratner}. Note that $ \nu \in (0,1] $ for all $ 0 \leq
\lambda <1/4 $.

\begin{lemma}
  \label{cor:estimate_ynm} For $ \nu \in (0,1] $ define
\[
\widetilde{f}_{n,m}(t) = 4(\nu -1) Y_{n,m}'(t) + 2 \imath (m+n)(2\nu -1) Y_{n,m}(t) + 4 \frac{(m-n)^2}{t} Y_{n,m}(t),
\]
and 
  \begin{equation}
    \begin{split}
    \label{eq:def_anm} A_{n,m}= \frac{1}{2\nu} \Big(& 5Y_{n,m}'(1) + (4 \imath (m+n) +1+\nu) Y_{n,m}(1) \\
&- \int_1^\infty r^
    {-1-\nu} \widetilde{f}_{n,m} (r) \,\D r\Big).
    \end{split}
  \end{equation}
  Then, for every $ t \geq 1 $, we have
  \[
    \abs{Y_{n,m}(t) - A_{n,m} t^{-1+\nu}} \leq \frac{C}{\nu^2}\|
    e_n\|_{W^3} \, \|e_m\|_{W^3} \, t^{-1}.
  \]
\end{lemma}
\begin{proof}
  The function $\widetilde{f}_{n,m}$ is uniformly bounded over $[1,\infty)$, 
  hence $A_{n,m}$ is well-defined.

  We consider the initial value problem given by the ODE~\eqref{eq:diffeq_ynm} 
  with the initial conditions $y(1) = Y_{n,m}(1) = \langle \ub^1.e_n,e_m \rangle$ 
  and $y'(1) = Y_{n,m}'(1) = \langle \ub^1.Ue_n,e_m \rangle$.
  Its solution is given by the formula
  \[
    \begin{split}
    Y_{n,m}(t) = &\frac{t^{-1+\nu}}{2\nu} \left( \int_1^t r^{-\nu}
    f_{n,m}(r)\, \D r  + (1+\nu) Y_{n,m}(1) + Y_{n,m}'(1)\right) \\
   & + \frac{t^{-1-\nu}}{2\nu} \left(- \int_1^t r^\nu f_{n,m}(r)\, \D r + (\nu-1) Y_{n,m}(1) - Y_{n,m}'(1)\right) .
    \end{split}
  \] 
Since $\nu \in (0,1]$, the function $r^{-\nu}f_{n,m}(r)$ is in $L^1([1,\infty))$. 
Thus, integrating by parts the terms $4y''$ and $-4\imath(m+n)y'(t)$ appearing in the 
definition \eqref{eq:def_of_f} of $f_{n,m}$, we can write
\[
    \begin{split}
 \int_1^t r^{-\nu} f_{n,m}(r)\, \D r =&\int_1^\infty r^{-\nu}  f_{n,m}(r)\, \D r 
-\int_t^\infty r^{-\nu}  f_{n,m}(r)\, \D r \\
=&4  y'(1) + 4\imath (m+n) y(1) - \int_1^\infty r^
    {-1-\nu} \widetilde{f}_{n,m} (r) \,\D r \\
&-\int_t^\infty r^{-\nu}  f_{n,m}(r)\, \D r .
    \end{split}
\]
Substituting in the expression for $Y_{n,m}(t)$, we obtain the inequality
  \begin{equation}
    \label{eq:bound_ynm}
    \begin{split}
& \abs{Y_{n,m}(t) - t^{-1+\nu} A_{n,m} } \leq t^{-1+\nu} 
 \abs {\, \int_t^{\infty}\frac{r^{-\nu}f_{n,m}(r)}{2\nu} \, \D r\,} \\
& \qquad  + t^{-1-\nu} \abs{\, -
      \int_0^t \frac{r^{\nu}f_{n,m}(r)}{2\nu}\, \D r \, + (\nu-1) Y_{n,m}(1) - Y_{n,m}'(1)}. 
    \end{split}
  \end{equation}
Using the bound on $|f_{n,m}|$ given by Proposition~\ref{prop:ynm}, 
we conclude that the two summands in the right-hand side of~\eqref{eq:bound_ynm}
  are both bounded by
  \begin{equation*}
    \frac{C}{2\nu^2}\|e_n\|_{W^3} \, \|e_m\|_{W^3} \, t^
    {-1}.
  \end{equation*}
  This completes the proof.
\end{proof}

We finish this section with a corollary that extends the asymptotics 
of Theorem \ref{thm:strong_ratner} to all $W^3$ vectors.

\begin{corollary}\label{cor:strong_Ratner} 
There exists a constant $C\geq 1$ such that the following holds.
Let $v,w \in W^3(H)$, and let 
\[
    \begin{split}
F_{v,w}(t) = &4(\nu -1) \langle \ubt. Uv,w\rangle + 2(2\nu -1) (\langle \ubt. \Theta v,w\rangle - \langle \ubt. v, \Theta w\rangle) \\
&- \frac{4}{t} (\langle \ubt. \Theta^2 v,w\rangle + \langle \ubt. v, \Theta^2 w\rangle  + 2\langle \ubt. \Theta v, \Theta w\rangle),
    \end{split}
\]
and 
  \begin{equation}
    \begin{split}
    \label{eq:def_avw} A(v,w)= &\frac{1}{2\nu} \Big(5 \langle \ub^1. U v,w\rangle + 4(\langle \ub^1. \Theta v,w\rangle - \langle \ub^1. v, \Theta w\rangle) \\
& +(1+\nu)\langle \ub^1. v,w\rangle  - \frac{1}{2\nu} \int_1^\infty r^{-1-\nu} F_{v,w} (r) \,\D r \Big).
    \end{split}
  \end{equation}
For every $t\geq 1$ we have
\[
\abs{\langle \ubt.v,w \rangle - t^{-1+\nu} A(v,w)} \leq \frac{C}{\nu^2}\|
    v\|_{W^3} \, \|w\|_{W^3} \, t^{-1}.
\]
\end{corollary}
\begin{proof}
By the representation theory of compact Abelian groups, $H$ admits an 
orthonormal basis of analytic eigenvectors $e_n$ of $\Theta$ with eigenvalue $\imath n$, 
where $n \in \Z$. 

Let us write $v=\sum_{n \in \Z} v_n e_n$ and $w=\sum_{n \in \Z} w_n e_n$.
From the definition \eqref{eq:def_anm} and the bound in Proposition 
\ref{prop:ynm}, we deduce that 
\[
|A_{n,m}|\leq C \|e_n\|_{W^3} \, \|e_m\|_{W^3}, 
\]
so that the Cauchy-Schwarz inequality yields
\[
\sum_{n,m \in \Z} \abs{v_n \, w_n \, A_{n,m}} \leq C \|v\|_{W^3} \, \|w\|_{W^3}.
\]
This proves that 
\[
A(v,w) = \sum_{n,m \in \Z} v_n \, w_n \, A_{n,m}
\]
is well-defined and the conclusion follows.
\end{proof}

%%%%%%%%%%

\section{Harmonic analysis on Abelian covers}%
\label{sec:spectrum-an-abelian}

\subsection{Characters and harmonic one-forms}

Let us recall that the Galois group $ \Gal = \Gamma / \Gamma_0 $ of the
cover $ p \colon S_0 \to S $, as well as of $ \Mc = T^1(\Sc) \to M = T^1
(S) $, is a finitely generated Abelian group, which is isomorphic to $
\Z^d $ for some $ 1 \leq d \leq 2g $, where $ g $ is the genus of $ S $.
The group $ \Gal $ acts on $ \Mc $ by left-translations, that is, for
any $ [\gamma] = \gamma \Gammac $ and any $ x = \Gammac g \in \Mc $ we
have $ [\gamma].x= \Gammac \gamma g $.  Similarly, $ \Gal $ acts on $
\Sc $ by $ [\gamma].  j(x) = j([\gamma].x) $, where $ j \colon \Mc \to
\Sc $ is the canonical projection $ j(\Gammac g) = \Gammac g \PSO(2,\R) $.

Let us denote by $ \Galhat $ the dual of $ \Gal $, namely the group of
characters $ \chi \colon \Gal \to U(1) $, which we identify with the
group of homomorphisms $ \chi \colon \Gamma \to U(1) $ such that $ \ker(\chi)
\supseteq \Gammac $. We now describe its Lie algebra in terms of
harmonic one-forms.

Recall that every cohomology class has a canonical representative which
is a harmonic one-form, namely a one-form $ \omega $ such that $ \Delta
\omega =0 $. %, where $\Delta$ denotes the Laplacian
We denote by $ \harm $ the $ d $-dimensional vector space of harmonic
one-forms on $ S $ representing cohomology classes with vanishing
periods on the cycles represented by $ \Gammac $.

Let $ x \in M $ be fixed.  For every $ \omega \in \harm $, define
\begin{equation*}
  \begin{split}
    \chi_\omega \colon \Gal &\to U(1) \subset \C \\
    [\gamma] &\mapsto \exp\left( 2 \pi \imath \int_{x}^{[\gamma].x} j^{\ast}
    p^{\ast} \omega\right),
  \end{split}
\end{equation*}
where the integral is taken along a path in $ \Mc $ from $ x $ to $ [\gamma].x
$. The value $ \chi_\omega([\gamma]) $ does not depend on the path
chosen. Indeed, since $ \omega $ vanishes on the cycles represented by
elements in $ \Gammac $, the one-form $ p^{\ast}\omega $ is exact on $
\Sc $ and vanishes on the fibers of the projection $ j \colon \Mc \to
\Sc $. Moreover, for the same reason, $ \chi_\omega $ does not depend on
the choice of the point $ x $.

\begin{lemma}
  \label{lem:Abelian_covers:1} For every $ \omega \in \harm $, we have $
  \chi_\omega \in \Galhat $. Moreover, the space $ \harm $ is isomorphic
  to the Lie algebra $ \Lie(\Galhat) $ of the group $ \Galhat $, the map
  $ \omega \mapsto \chi_\omega $ coincides with the exponential map of $
  \Galhat $ and induces an isomorphism between $ \Galhat $ and $ \harm /
  \harm(\Z) $, where $ \harm(\Z) $ consists of harmonic one-forms
  representing elements in $ H^1(S, \Z) $.
\end{lemma}
\begin{proof}
  For every $ [\gamma_1], [\gamma_2] \in \Gal $ and $ \omega \in \harm $,
  we have
  \[
    \begin{split}
      \chi_\omega([\gamma_1] \cdot [\gamma_2]) &= \exp \left(2 \pi
      \imath \int^{[\gamma_1]\cdot[\gamma_2].  x}_x j^\ast p^\ast\omega\right)
      \\
      &= \exp\left(2 \pi \imath \int^{[\gamma_2].x}_x j^\ast p^\ast\omega\right)
      \exp \left(2 \pi \imath \int^{[\gamma_1]\cdot [\gamma_2].x}_{[\gamma_2].x}
      j^\ast p^\ast \omega\right)\\
      &=\exp \left(2 \pi \imath \int^{[\gamma_2].x}_x j^\ast p^\ast
      \omega\right) \exp \left(2\pi \imath \int^{[\gamma_1].x'}_{x'} j^\ast
      p^\ast \omega\right) \\
      &=\chi_\omega ([\gamma_1]) \chi_\omega ([\gamma_2]),
    \end{split}
  \] where $ x' = [\gamma_2].x $.  This shows that $ \chi_\omega $ is a
  character of $ \Gal $. It is easy to see that, for every $ \omega_1,
  \omega_2 \in \harm $, we have $ \chi_{\omega_1 +\omega_2 } = \chi_{\omega_1}
  \chi_{\omega_1} $, proving that the map $ \omega \mapsto \chi_\omega $
  is a homomorphism of the additive Abelian group $ \harm $ into
  multiplicative Abelian group $ \Gal $. By the defition of the space $
  \harm $, the differential of this map at the origin $ \omega\mapsto
  \left([\gamma] \mapsto 2 \pi \imath \int^{[\gamma].x}_{x} j^\ast p^\ast\omega\right)
  $ is an isomorphism, thus concluding the proof.
\end{proof}

\subsection{The line bundle associated to a character}%
\label{sec:line_bundles}

Let us fix a fundamental domain $ \fund \subset \Mc $ for the action of $
\Gal $ on $ \Mc $.  Any function $ f \colon M \to \C $ can be seen as a
section of the trivial bundle $ M_{\one}=M \times \C $, which could be
defined as
\[
  M_{\one} = \Gal \backslash (\Mc \times \C),
\] where $ \Gal $ acts on $ \Mc \times \C $ by $ [\gamma].(x,z) = ([\gamma]^
{-1}.x,z) $. We generalize the previous definition by allowing a
possibly non-trivial holonomy; in other words, given a character $ \chi
\in \Galhat $, we define the line bundle $ M_{\chi} \to M $ associated
to $ \chi $ by
\[
  M_{\chi} = \Gal \backslash (\Mc \times \C),
\] where $ \Gal $ acts on $ \Mc \times \C $ by $ [\gamma].(x,z) = ([\gamma]^
{-1}.x, \chi([\gamma]) z) $.

We denote by $ L^2(M_{\chi}) $ the Hilbert space completion of the space
of continuous sections of $ M_{\chi} $ with respect to the inner product
\[
  \langle f_1, f_2 \rangle_{L^2(M_\chi)} := \int_{M} f_1 \, \overline{f_2}
  \diff
  \vol.
\] We remark that $ f_1 \, \overline{f_2} $ is a well-defined function
on $ M $ for any two sections $ f_1, f_2 $ of~$ M_{\chi} $.
Equivalently, we could define a section of $ M_\chi $ as a function $ f $
on $ \Mc $ which satisfies the condition $ f([\gamma]^{-1}.x) = \chi([\gamma])
f(x) $, and $ L^2(M_\chi) $ as the space of such measurable functions
whose square norm $ \int_{\fund} |f|^2 \diff
\vol $ is finite. By the same token, for every positive integer $ s $,
we define the $ W^{2s} $ Sobolev norm of $ f $ as
\[
  \|f\|_{W^{2s}(M_\chi)} := \int_\fund \Delta^s f \diff
  \vol,
\] and the Sobolev norm for other real values of $ s $ can be obtained
by interpolation.

Let us give an explicit example, which will be relevant in the
following. Let us fix $ x_0 \in \fund $.  Given $ \chi = \chi_{\omega}
\in \Galhat $, where $ \omega \in \harm $, let us define
\begin{equation}\label{eq:definition_G_omega}
  G_\omega(x) = \exp \left( 2 \pi \imath \int_{x}^{x_0} j^{\ast} p^{\ast}
  \omega\right).
\end{equation}
Then,
\begin{equation*}
  \begin{split}
    G_\omega([\gamma]^{-1}.x) &= \exp \left( 2 \pi \imath \int_{[\gamma]^
    {-1}.x}^{x_0} j^{\ast} p^{\ast} \omega\right) \\
    &= \exp \left( 2 \pi \imath \int_{[\gamma]^{-1}.x}^x j^{\ast} p^{\ast}
    \omega\right) \, \exp \left( 2 \pi \imath \int_{x}^{x_0} j^{\ast} p^
    {\ast} \omega\right)\\
    &= \exp \left( 2 \pi \imath \int_{x}^{[\gamma].x} j^{\ast} p^{\ast}
    \omega\right) \, G_\omega(x)= \chi_\omega([\gamma]) \, G_\omega(x),
  \end{split}
\end{equation*}
which proves that $ G_\omega $ is a section of $ \Mchio $.  Moreover,
\[
  \int_\fund |G_\omega|^2 \diff
  \vol =
  \vol(\fund) < \infty,
\] so that $ G_\omega \in L^2(\Mchio) $. Furthermore, it follows 
from \eqref{eq:definition_G_omega} that $
G_\omega $ is real-analytic as a function of $ \omega \in \harm $.

Let us compute its derivatives.  We identify $ x \in \Mc $ with the
point $ (z,v) \in T^1(\Sc) $.  Easy computations show
\begin{equation}
  \label{eq:derivatives_G_omega}
  \begin{split}
    (\Theta G_\omega)(x) &= 2 \pi \imath \, \Theta \left(\int_{x}^{x_0}
    j^{\ast} p^{\ast} \omega \right) G_\omega(x) = 0, \\
    (X G_\omega)(x) &= 2 \pi \imath \, X\left(\int_{x}^{x_0} j^{\ast} p^
    {\ast} \omega \right) G_\omega(x) = -2 \pi \imath \, (p^{\ast}
    \omega)_z(v)\, G_\omega(x),\\
    (Y G_\omega)(x) &= 2 \pi \imath \, Y\left(\int_{x}^{x_0} j^{\ast} p^
    {\ast} \omega \right) G_\omega(x) = -2 \pi \imath \, (p^{\ast}
    \omega)_z(v^{\perp})\, G_\omega(x),
  \end{split}
\end{equation}
where $ v^{\perp} $ is the vector obtained by rotating $ v $ by $ \pi/2 $
in the clockwise direction.

Let now $ s \in L^2(M_\chi) $.  Then, the function $ G_\omega \cdot s $
satisfies
\begin{equation*}
  \begin{split}
    (G_\omega \cdot s)([\gamma]^{-1}.x) &= G_\omega([\gamma]^{-1}.x) \,
    s([\gamma]^{-1}.x) = \chi_\omega ([\gamma]) \, \chi([\gamma]) \, G_\omega
    (x)\, s(x) \\
    &=(\chi \chi_{\omega})([\gamma]) \, (G_\omega \cdot s)(x),
  \end{split}
\end{equation*}
so that it is a section of $ \Mcchio $.  Moreover, it is also clear that
\[
  \|G_\omega \cdot s \|_{L^2(\Mcchio)} = \|s\|_{L^2(M_\chi)}.
\] We obtain the following lemma.
\begin{lemma}
  \label{lem:Abelian_covers:2} The map
  \[
    I_\omega \colon s \mapsto G_\omega \cdot s
  \] is a surjective linear isometry between $ L^2(M_\chi) $ and $ L^2(\Mcchio)
  $. Moreover, for every $ r >0 $, the map $ I_\omega $ is a linear
  isomorphism between the Sobolev spaces $ W^r(M_\chi) $ and $ W^r(\Mcchio)
  $.
\end{lemma}
\begin{proof}
  Since we have already verified the first part, it remains to show that
  $ I_\omega $ is a bounded map between $ W^s(M_\chi) $ and $ W^s(\Mcchio)
  $.

  For every $ j \geq 0 $, let us denote $ \|\omega\|_{\mathscr{C}^j(\fund)}
  $ the $ \mathscr{C}^j $-norm of $ \omega \colon \fund \to T^\ast\fund $.
  Using~\eqref{eq:derivatives_G_omega} we can bound
  \[
    \| \Theta^j G_\omega\|_{L^\infty(\fund)}, \, \| X^j G_\omega\|_{L^\infty
    (\fund)} , \, \| Y^j G_\omega\|_{L^\infty(\fund)} \leq (2\pi)^{j-1}
    \|\omega\|_{\mathscr{C}^j(\fund)}.
  \] Therefore, for any integer $ r>0 $ and every $ j_1, j_2, j_3 \in \Z_
  {\geq 0} $ such that $ j_1+j_2+j_3 = r $, we have
  \[
    \begin{split}
      & \|X^{j_1}Y^{j_2}\Theta^{j_3}(G_\omega \cdot s)\|_{L^2(\Mcchio)}
      \\
      & \qquad \leq (2\pi)^r \|\omega\|_{\mathscr{C}^r(\fund)} \left(
      \sum_{j_1+j_2+j_3 \leq r}\|X^{j_1} Y^{j_2}\Theta^{j_3}s\|_{L^2(M_\chi)}\right)
      \\
      & \qquad \leq C_{\omega}(r) \|s\|_{W^r(M_\chi)},
    \end{split}
  \] for some constant $ C_\omega(r)>0 $ depending only on $ \omega $
  and $ r $. This completes the proof for integers $ s $, the general
  case follows from interpolation.
\end{proof}

\subsection{Rapidly decaying functions}%
\label{sec:rap_dec_fns}

We now define the space of functions we are interested in.  Let us fix a
norm $ \| \cdot \| $ in the first homology $ H_1(S,\R) $.  Since $ \Gal $
is a rank-$ d $ sublattice of $ H_1(S,\Z) $, we have that
\[
  \sum_{[\gamma] \in \Gal} \|[\gamma]\|^{-d-\varepsilon} < \infty,
  \qquad \text{ for every\ } \varepsilon >0.
\] We define the space $ \rapid $ of what we call \emph{rapidly decaying
functions} as the subspace of $ \mathscr{C}^\infty(\Mc) $ consisting of
smooth functions $ f $ on $ \Mc $ such that
\begin{multline*}
 \sum_{[\gamma] \in \Gal} \|[\gamma]\|^{k +d +\varepsilon} \int_{[\gamma].\fund}
  |D^n f(x)|^2 \diff 
  \vol(x) < \infty,  \\
\text{ for every\ } k,n \in \Z_{\geq 0} \text{\
  and\ } \varepsilon >0,
  \end{multline*}
where $ |D^n f(x)|$ is any norm of the $n$-th derivative 
$D^nf(x) \colon T_xM_0 \otimes \cdots \otimes T_xM_0 \to \C$ 
of $f$ at $x \in M_0$. 
In particular, we observe that $ \rapid $ contains the space $
\mathscr{C}^\infty_c(\Mc) $ of smooth, compactly supported functions
since, for these functions, only finitely many terms in the sum above
are non-zero.

For any $ \chi \in \Galhat $ and $ f \in \rapid $, let us define
\[
  \widehat{f}(x \mid \chi) := \sum_{[\gamma] \in \Gal} f([\gamma].x)
  \chi([\gamma]).
\] By the Cauchy-Schwarz Inequality, for every $ n \in \Z_{\geq 0} $ and
$ \varepsilon >0 $ we have
\[
  \begin{split}
    \left\lvert D^n \widehat{f}( x \mid \chi) \right\rvert^2 &=
    \left\lvert \sum_{[\gamma] \in \Gal} D^n f([\gamma].x) \chi([\gamma])
    \right\rvert^2 \\
    &\leq \left( \sum_{[\gamma] \in \Gal} | D^n f([\gamma].x)|^2 \,
    \|[\gamma]\|^{d+\varepsilon} \right) \left( \sum_{[\gamma] \in \Gal}
    \|[\gamma]\|^{-d-\varepsilon} \right),
  \end{split}
\] and, from the definition of $ \rapid $, it follows that
\[
  \begin{split}
    &\int_\fund |D^n \widehat{f}(x \mid \chi)|^2 \diff
    \vol \\
    &\qquad \leq \left( \sum_{[\gamma] \in \Gal} \|[\gamma]\|^{d +\varepsilon}
    \int_{[\gamma].\fund} |D^n f|^2 \diff
    \vol \right) \left( \sum_{[\gamma] \in \Gal} \|[\gamma]\|^{-d-\varepsilon}
    \right) < \infty.
  \end{split}
\] Moreover, for every $ [\gamma]\in \Gal $, we have
\[
  \begin{split}
    \widehat{f}([\gamma]^{-1}.x \mid \chi) &= \sum_{[\sigma] \in \Gal} f
    ([\sigma]. ([\gamma]^{-1}.x)) \chi([\sigma]) = \sum_{[\sigma] \in
    \Gal} f([\sigma].x) \chi([\sigma \gamma]) \\
    &= \chi([\gamma]) \sum_{[\sigma] \in \Gal} f([\sigma].x) \chi([\sigma])
    = \chi([\gamma]) \widehat{f}(x \mid \chi).
  \end{split}
\] Thus, we obtain a well-defined map
\[
  \pi_\chi \colon \rapid \to \bigcap_{r\geq 0} W^r(M_\chi) \subset L^2(M_\chi).
\] Furthermore, the map $ \pi_\chi $ is $ G $-equivariant:  if we set $
(\rho(g)f)(x)=f(xg) $, then we have $ \pi_\chi(\rho(g)f)(x) = \widehat{f}
(xg \mid \chi) $.

By Lemma~\ref{lem:Abelian_covers:1}, we can identify $ \Galhat $ with the torus $
\harm / \harm(\Z) $.  Parseval Identity yields
\[
  \|f\|^2_{L^2(\Mc)} = \int_\fund \sum_{[\gamma] \in \Gal} |f([\gamma].x)|^2
  = \int_{\Galhat} \left( \int_\fund |\widehat{f}(x \mid \chi_\omega)|^2
  \diff
  \vol \right) \diff \omega,
\] hence, for any $ f_1, f_2 \in \rapid $, it follows that
\[
  \langle f_1, f_2 \rangle = \int_{\Galhat} \langle \pi_{\chi_\omega}(f_1),
  \pi_{\chi_\omega}(f_2) \rangle_{L^2(\Mchio)} \diff \omega.
\] By density of $ \mathscr{C}^\infty_c(\Mc) $, and hence of $ \rapid $,
in $ L^2(\Mc) $, we deduce the following lemma, which the reader can
compare with, e.g.,
\cite[Theorem 4.1]{Dei}.
\begin{lemma}
  \label{lem:L2_direct_integral} We have the following $ G $-equivariant
  isometric direct integral decomposition:
  \[
    L^2(\Mc) = \int_{\Galhat} L^2(\Mchio) \diff \omega, \qquad f = \int_
    {\Galhat} \pi_{\chi_\omega}(f) \diff \omega.
  \]
\end{lemma}

\subsection{Twisted $ G $-actions}

Let $ \one = \chi_0 \in \Galhat $ be the trivial character.  For any $ f
\in \rapid $ and $ \omega \in \harm $, we have
\[
  \begin{split}
    (\pi_\one \circ I_{-\omega})(f)(x) &= \pi_\one( G_{-\omega} \cdot f)
    (x) = \sum_{[\gamma] \in \Gal} (G_{-\omega} \cdot f) ([\gamma].x) \\
    &= \sum_{[\gamma] \in \Gal} \chi_{\omega}([\gamma]) G_{-\omega}(x) f
    ([\gamma].x) = (I_{-\omega} \circ \pi_{\chi_\omega})(f)(x).
  \end{split}
\] In particular, by Lemma~\ref{lem:Abelian_covers:2}, we deduce that for every $ f_1,f_2 \in
\rapid $ and $ \chi=\chi_\omega \in \Galhat $,
\begin{equation}
  \label{eq:to_M_one} \langle \pi_{\chi_\omega}(f_1), \pi_{\chi_\omega}(f_2)
  \rangle_{L^2(\Mchio)} = \langle \pi_{\one}(G_{-\omega} \cdot f_1), \pi_
  {\one}(G_{-\omega} \cdot f_2) \rangle_{L^2(M)}.
\end{equation}
In the following, for any $ f \in \rapid $, we will write $ \Omega_\omega
(f) = \pi_{\one}(G_{-\omega} \cdot f) \in L^2(M) $.

From Lemma~\ref{lem:L2_direct_integral} and~\eqref{eq:to_M_one}, it follows that we
can express the correlations of two functions $ f_1, f_2 \in \rapid $ as
an integral of correlations
\[
  \begin{split}
    \langle f_1, f_2\rangle_{L^2(\Mc)} &= \int_{\Galhat} \langle \pi_{\chi_\omega}
    (f_1), \pi_{\chi_\omega}(f_2) \rangle_{L^2(\Mchio)} \diff \omega \\
    &=\int_{\Galhat} \langle \Omega_\omega(f_1), \Omega_\omega(f_2)
    \rangle_{L^2(M)} \diff \omega
  \end{split}
\] on the same Hilbert space $ H= L^2(M) $. The corresponding unitary
representation $ \rho_\omega \colon G \to \mathcal{U}(H) $ of $ G=\PSL_2
(\R) $ on $ H $ is defined so that for all $ W \in \sl_2(\R) $ we have
\[
  W_\omega (\Omega_\omega(f)) = \Omega_\omega(Wf), \qquad \text{ where }
  \qquad W_\omega = \diff \rho_\omega(W).
\] In the following, we will denote by $ H_\omega $ the datum of the
Hilbert space $ H = L^2(M) $ together with the unitary representation $
\rho_\omega $.

Let us write explicit formulas for the generators $ \Theta_\omega, X_\omega,
Y_\omega $. For every $ W \in \sl_2(\R) $, since $ \pi_\one $ is $ G $-equivariant,
we have
\[
  \begin{split}
    \Omega_\omega(Wf) &= \pi_\one(G_{-\omega} \cdot Wf)= \pi_\one( W(G_{-\omega}
    \cdot f) - (WG_{-\omega}) \cdot f ) \\
    &= W(\Omega_\omega(f)) - \pi_\one( (WG_{-\omega}) \cdot f),
  \end{split}
\] thus, using~\eqref{eq:derivatives_G_omega} to compute the derivatives
of $ G_{-\omega} $ in the right-hand side above, we obtain
\begin{equation}
  \label{eq:generator_twisted_action}
  \begin{split}
    \Theta_\omega f(x) &= \Theta f(x), \\
    X_\omega f(x) &= Xf(x) -2 \pi \imath \, \omega_z(v)\, f(x),\\
    Y_\omega f(x) &= Yf(x)-2 \pi \imath \, \omega_z(v^{\perp})\, f(x),
  \end{split}
\end{equation}
for any $ f \in H $ and $ x = (z,v) \in M = T^1(S) $. %
In other words, we obtained the following lemma.
\begin{lemma}
  \label{lem:int_of_corrs} For every $ f_1,f_2 \in \rapid $, there exist
  two associated smooth family of vectors $ \omega \mapsto \Omega_\omega
  (f_j) \in H=L^2(M) $ for $ j=1,2 $ such that for every $ t \in \R $
  \[
    \langle f_1 \circ h_t, f_2 \rangle_{L^2(\Mc)} = \int_{\Galhat}
    \langle \rho_\omega (\exp(tU)).  \Omega_\omega(f_1) , \Omega_\omega(f_2)
    \rangle_{H} \diff \omega,
  \] where the differential $ \diff \rho_\omega $ of the action $ \rho_\omega
  \colon G \to \mathcal{U}(H) $ is defined by~\eqref{eq:generator_twisted_action}.
\end{lemma}
We now verify that, for every fixed $ t \in \R $, the integrand function
\[
  \omega \mapsto \langle \uot.  \Omega_\omega(f_1),
  \Omega_\omega (f_2) \rangle_{L^2(M)}
\] is smooth,
where we denote $\uot = \rho_\omega(\exp(tU))$.

\begin{lemma}
  \label{lem:Omega_f} Let $ f \in \rapid $ be fixed.  Then, $ \omega
  \mapsto \Omega_\omega(f) $ is a $ \mathscr{C}^\infty $ map from $
  \harm $ to $ W^r(H_{\omega}) $ for every $ r \geq 0 $.
\end{lemma}
\begin{proof}
  We prove that the map $ \omega \mapsto \Omega_{\omega}(f) $ from $
  \harm $ to $ H $ is $ \mathscr{C}^\infty $, the case $ r>0 $ can be
  proved in a similar way using~\eqref{eq:generator_twisted_action} and~\eqref{eq:derivatives_G_omega}.

  For every $ n \in \Z_{\geq 0} $, we have
  \[
    \left\lvert \frac{\diff^n}{(\diff \omega)^n} G_{-\omega}([\gamma].x)
    \right\rvert \leq (1+\|[\gamma]\|)^n.
  \] Therefore, the series
  \[
    \begin{split}
          \sum_{[\gamma] \in \Gal} &\left\| \frac{\diff^n}{(\diff \omega)^n} (G_
            {-\omega} \cdot f) ([\gamma].x) \right\|^2_{L^2(\Mchio)} \\
          & \quad\leq \sum_{%
    [\gamma] \in \Gal} (1+\|[\gamma]\|)^n \left\| f([\gamma].x) \right\|^2_
  {L^2(\Mchio)}
    \end{split}
  \] converges, since $ f $ is a rapidly decaying function according to
  \S~\ref{sec:rap_dec_fns}.  This proves that $ \Omega_\omega(f) $ is $ n $
  times differentiable with respect to $ \omega $ and
  \[
    \frac{\diff^n \Omega_\omega(f)}{(\diff \omega)^n}= \sum_{[\gamma]
    \in \Gal} \left(\frac{\diff^n}{(\diff \omega)^n} G_{-\omega} ([\gamma].x)\right)
    f([\gamma].x),
  \] which completes the proof.
\end{proof}

\begin{lemma}
  \label{lem:uot_smooth} Let $ \omega \mapsto f_\omega$ 
   be a family of smooth vectors in $W^r(H_\omega)$.
   Then, $\omega \mapsto \uot . f_\omega$ is a smooth map from $
   \harm $ to $ W^r(H_{\omega}) $, and for every $ n \in \Z_{\geq 0} $ and 
   for every $t\geq 1$, we have
  \[
    \left\| \frac{\diff^n}{(\diff \omega)^n} \uot.f_\omega
    \right\| \leq C_n(\omega) (1+\log t)^n \max_{k\leq n} \left\| 
    \frac{\diff^k f_\omega}{(\diff \omega)^k},
    \right\|
  \] 
  for some constant $C_n(\omega)$ independent of $t$.
\end{lemma}
\begin{proof}
From the definition of the twisted action $\rho_\omega$, we can write
\[
\uot. f_\omega = f_\omega \circ h_t \cdot \exp\left( 2 \pi \imath 
\int_{x}^{h_t(x)} j^{\ast} p^{\ast} \omega\right).
\]
Thus, there exists a constant $C_n$ and $k \leq n$ so that 
\[
    \begin{split}
 \left\| \frac{\diff^n}{(\diff \omega)^n} \uot.f_\omega
    \right\| & \leq C_n \left\| \frac{\diff^{n-k} f_\omega}{
    (\diff \omega)^{n-k}} \right\| \cdot 
\left\| \frac{\diff^k}{(\diff \omega)^k} 
\exp\left( 2 \pi \imath 
\int_{x}^{h_t(x)} j^{\ast} p^{\ast} \omega\right)
    \right\|_{\infty} \\
& \leq C_n  \left\| \frac{\diff^{n-k} f_\omega}{
    (\diff \omega)^{n-k}} \right\| \cdot  
\left\| \int_{x}^{h_t(x)} j^{\ast} p^{\ast} \omega 
    \right\|^k_{\infty}. 
    \end{split}
\]
It remains to bound the integral $\int_{x}^{h_t(x)} j^{\ast} p^{\ast} \omega$. 
The commutation relation between geodesic and horocycle flow tells us that 
$g_{\log t}(h_t(x)) = h_1(g_{\log t}(x))$, 
where $g_t$ denotes the geodesic flow.
Since $p^{\ast} \omega$ is exact on the cover $S_0$, Stokes' Theorem yields
\[
\int_{x}^{h_t(x)} j^{\ast} p^{\ast} \omega = 
\int_{x}^{g_{\log t}(x)} j^{\ast} p^{\ast} \omega +
\int_{g_{\log t}(x)}^{h_1(g_{\log t }(x))} j^{\ast} p^{\ast} \omega -
\int_{h_t(x)}^{g_{\log t}(h_t(x))} j^{\ast} p^{\ast} \omega.
\]
Since all paths in the right hand side above have length at most 
$1+\log t$, we deduce the bound
\[
\left\lvert \int_{x}^{h_t(x)} j^{\ast} p^{\ast} \omega 
    \right\rvert \leq 2 \|\omega\|_{\infty} (1+\log t),
\]
which completes the proof.
\end{proof}

In Section~\ref{sec:strong_ratner}, we strengthened Ratner's result to obtain
precise estimates of the asymptotics of the correlations above, depending on the 
unitary representation. 
We now need to understand the behaviour of the
spectrum of the associated Casimir operator, which is the subject of the
remainder of this section.

\subsection{The spectra around zero}%
\label{sec:spec_around_zero}

As above, we denote by $ H_\omega $ the space $ H=L^2(M) $ equipped with
the action $ \rho_\omega $ of $ G $. By the standard theory of unitary
representations of $ G $, every Hilbert space $ H_\omega $ decomposes
into irreducible representations $ H_{\omega, \lambda(\omega)} $, and on
each of these subspaces the Casimir operator
\[
  \square_\omega = -X^2_\omega -Y^2_\omega + \Theta^2_\omega
\] acts as multiplication by the real number $ \lambda(\omega) \in \R $.
If $ \lambda(\omega) >0 $, the space $ H_{\omega, \lambda(\omega)} $
contains an invariant vector by $ \Theta_\omega = \Theta $. Therefore,
in order to study the positive parameters $ \lambda(\omega) $, we can
study the eigenvalues of the twisted Laplace-Beltrami operator $ \Delta_
{S,\omega} = \diff_\omega \, \diff_\omega^\ast + \diff_\omega^\ast \,
\diff_\omega $.  Here
\[
  \diff_\omega = I_\omega \circ \diff \circ I_{-\omega}= \diff + 2 \pi
  \imath \, \omega \wedge
\] is the twisted differential operator and $ \diff_\omega^\ast $ its $
L^2 $ adjoint (see the proof of Proposition~\ref{thm:deriv_eigenvalues} below).

Recall that the maps $ I_{\pm \omega} $ are given by multiplication by $
G_{\pm \omega} $, which is a holomorphic function in the parameter $
\omega \in \harm $.  Thus, the family of elliptic operators $ \omega
\mapsto \Delta_{S,\omega} $ is holomorphic.  By classical results in
perturbation theory of selfadjoint operators (see, e.g.,
\cite[p.~392]{Kat}) the eigenvalues $ 0\leq \lambda_0(\omega) \leq
\lambda_1(\omega) \leq \cdots $ of $ \Delta_{S,\omega} $ are real
analytic functions of $ \omega $ in a neighborhood of $ 0 $.

We will now focus on the first eigenvalue $ \lambda_0(\omega) $ and we
will show that $ \lambda_0(\omega) \geq 0 $ with equality if and only if
$ \omega=0 $. Recall that we denoted by $ g $ the genus of the surface $
S $.

\begin{proposition}
  \label{thm:deriv_eigenvalues} For every $ \eta, \zeta \in \harm $, we
  have
  \[
    D\lambda_0(0)(\eta) = 0, \qquad \text{and} \qquad D^2\lambda_0(0)(\eta,
    \zeta) = \frac{2 \pi}{g-1} \langle \eta, \zeta \rangle.
  \] In particular, the bilinear form $ D^2\lambda_0(0) $ is positive
  definite.
\end{proposition}
\begin{proof}
  Given a real valued harmonic one-form $ \omega\in \harm $ we consider
  the twisted differential operator defined above
  \[
    \diff_\omega :  \Omega^*(S) \to \Omega^*(S), \quad \diff_\omega:=
    \diff + 2\pi \,\imath \,\omega\wedge.
  \] As $ \omega $ is closed, $ \diff_\omega ^2=0 $.  The $ L^2 $
  adjoint of $ \diff_\omega $ on $ \Omega^*(S) $ is given by
  \[
    \diff^*_\omega= \diff^* + (2\pi \,\imath\,\omega \wedge)^\dag= \diff^*
    - *(2\pi \,\imath\,\omega\wedge *), \quad \text{on } \Omega^*(S)
  \] where $ \diff^*= -*\diff* $ is the usual co-differential and $ * $
  is usual Hodge operator verifying
  \[
    \langle \eta,\theta\rangle = \int_S \eta \wedge *\bar \theta, \quad\text
    {for all } {\eta},{\theta}\in \Omega^\ell(S).
  \] We can now define the twisted Laplace-Beltrami operator on $ \Omega^*
  (S) $ by setting
  \[
    \Delta^{\phantom{*}}_\omega= \diff_\omega^{\phantom{*}} \diff^*_\omega+
    \diff^*_\omega \diff^{\phantom{*}}_\omega.
  \] Let $ f_\omega\in L^2(S) $, $ \omega\in U\subset \harm $, be any
  smooth family of normalised eigenfunctions of the Laplace-Beltrami
  operator $ \Delta_\omega = \diff^*_\omega \diff^{\phantom{*}}_\omega $
  with simple eigenvalue $ \lambda_\omega $:
  \[
    \Delta_\omega f_\omega = \lambda_\omega f_\omega, \quad \forall\omega:
    \norm{f_\omega}=1,
  \] and set $ h(\omega) = \langle\Delta_\omega f_\omega, f_\omega\rangle=
  \langle \diff_\omega f_\omega,\diff_\omega f_\omega\rangle $.
  Differentiating the identity $ h(\omega)= \lambda_\omega $ we have,
  for any $ \eta\in \mathcal H $
  \[
    D \lambda_\omega (\eta) = 2 \Re\langle 2 \pi \imath\,f_\omega\eta,
    \diff_\omega f_\omega\rangle
  \] and
  \[
    \begin{split}
      \frac 1 2 D^2\lambda_\omega (\eta,\zeta) &= \Re\langle 2 \pi
      \imath\,f_\omega\eta,2 \pi \imath\,f_\omega\zeta\rangle + \Re\langle
      2 \pi \imath\,Df_\omega(\zeta) \eta, \diff_\omega f_\omega\rangle
      \\
      &\qquad+ \Re\langle 2 \pi \imath\,f_\omega \eta, \diff_\omega Df_\omega
      (\zeta)\rangle.
    \end{split}
  \] In particular when $ \omega=0 $ we can take $ f_0=(4 \pi (g-1))^{-1/2}
  $, where $ 4 \pi (g-1) $ is the area of $ S $.  Then the second terms
  of the above identity vanish because $ \diff_0 f_0=0 $.  The third
  term also vanishes because $ \langle 2 \pi \imath \eta, \diff Df_0(\zeta)\rangle=
  \langle 2 \pi \imath \, \diff^*\eta, Df_0(\zeta)\rangle $ and the
  harmonicity of $ \eta $ implies that $ \eta $ is co-closed, i.e.\ $
  \diff^* \eta=0 $.  Then the latter identity becomes
  \[
    \frac 1 2 D^2\lambda_0 (\eta,\zeta) = \frac{ \pi}{g-1} \langle \eta,\zeta\rangle
  \] since $ \eta $ and $ \zeta $ are real.
\end{proof}

We will also need the following result on the eigenvalues $ \lambda_j(\omega)
$.
\begin{lemma}
  \label{thm:other_eigenvalues} The value $ 0=\lambda_0(0) $ is a global
  minimum for $ \lambda_j(\omega) $.  Moreover, there exists $ \delta >0
  $ such that $ \lambda_j (\omega) > \delta $ for all $ j \geq 1 $ and
  all $ \omega \in \harm $.
\end{lemma}
\begin{proof}
  Both assertions follow from the min-max principle, see
  \cite[p.  289]{PhSa}.
\end{proof}

%%%%%%%%%%

\section{Infinite volume mixing asymptotics}%
\label{sec:proof_main_thm}

This section is devoted to the proof of Theorem~\ref{thm:main1}. We start with a technical result.

\subsection{A \lq\lq stationary phase\rq\rq\ estimate}

The following lemma is a standard estimate of a certain type of
integrals, whose asymptotics is controlled by the behaviour of the
exponent in the integrand function near the stationary point. We present
a proof for the sake of completeness, see also
\cite[Lemma 2.3]{PhSa}.
\begin{lemma}
  \label{thm:stationary_phase} Let $ V $ be a $ d $-dimensional vector
  space and let $ U $ be an open neighborhood of $ 0 $.  Let $ v \colon
  U \to \R $ be a smooth function such that $ v(0) = 0 $ and $ v(\xi) <0
  $ for $ \xi \neq 0 $.  Assume that $ H = - D^2v(0) $ is positive
  definite. Let $ a \colon U \to \R $ be a $ \mathscr{C}^\infty $
  function.  Then, there exist constants $ (c_j)_{j \in \N} $ such that
  for all $ N \geq 1 $ we have
  \[
    \int_U e^{T v(\xi)} a(\xi) \diff \xi = \frac{(2 \pi)^{\frac{d}{2}}}{T^
    {\frac{d}{2}} \sqrt{\det H}} a(0) + \sum_{j=1}^N \frac{c_j}{T^{j+\frac
    {d}{2}}} + o(T^{-N-\frac{d}{2}}).
  \]
\end{lemma}
\begin{proof}
  By assumption, 0 is a non-degenerate critical point of $ v $.
  Therefore, by the Morse Lemma, there exist open neighbourhoods $ U',W'
  \subset U $ of 0 and a smooth diffeomorphism $ \rho \colon W' \to U' $,
  $ \rho(\widetilde \vartheta) = \xi $, with $ \rho(0)=0 $ and $ \det D\rho
  (0)=1 $, such that
  \[
    v(\xi) = v(\rho(\widetilde \vartheta)) = \frac{1}{2} D^2v(0)( {\widetilde\vartheta},
    \widetilde \vartheta).
  \] Without loss of generality, up to choosing a smaller $ W' $, we can
  assume that $ \det D\rho(\widetilde \vartheta) >0 $ on $ W' $.
  Applying $ \rho $, followed by another linear change of variable, we
  get
  \begin{equation*}
    \begin{split}
      \int_{U'} e^{T v(\xi)} a(\xi) \diff \xi &= \int_{W'} e^{-\frac{T}{2}
      H({\widetilde\vartheta}, \widetilde \vartheta)}(a\circ\rho)(\widetilde\vartheta)
      \det D\rho(\widetilde \vartheta) \diff \widetilde \vartheta \\
      &=\frac{1}{\sqrt{\det H}} \int_W e^{-\frac{T}{2}(\vartheta_1^2 +
      \cdots + \vartheta_d^2)} b(\vartheta) \diff \vartheta,
    \end{split}
  \end{equation*}
  where $ W \subset U $ is an open neighbourhood of 0 and $ b $ is a
  smooth function such that $ b(0)=a(0) $. Therefore, we can find $
  \alpha, \delta >0 $, which depend only on $ v $ and $ a $, such that
  \[
    \int_U e^{T v(\xi)} a(\xi) \diff \xi = \frac{1}{\sqrt{\det H}} \int_
    {\|\vartheta\|\leq \delta} e^{-\frac{T}{2}\|\vartheta\|^2} b(\vartheta)
    \diff \vartheta + O(e^{-\alpha T}),
  \] where we used the fact that $ v(\xi) <0 $ for all $ \xi \neq 0 $.

  In order to complete the proof, we are left to estimate the term
  \begin{equation}
    \label{eq:to_estimate_sf}
    \begin{split}
      &\frac{1}{\sqrt{\det H}} \int_{\|\vartheta\|\leq \delta} e^{-\frac
      {T}{2}\|\vartheta\|^2} b(\vartheta) \diff \vartheta \\
      & \qquad = \frac{1}{T^{\frac{d}{2}}\sqrt{\det H}} \int_{\|\vartheta\|\leq
      \delta\sqrt{T}} e^{-\frac{1}{2}\|\vartheta\|^2} b\left(\frac{\vartheta}
      {\sqrt{T}} \right) \diff \vartheta.
    \end{split}
  \end{equation}
  For any $ N \geq 1 $, we can write the Taylor expansion of $ b $ as
  \[
    b\left(\frac{\vartheta}{\sqrt{T}} \right) = b(0) + \sum_{|k| \leq N}
    \frac{D^kb(0)}{k!} \frac{\vartheta^k}{T^{\frac{|k|}{2}}} + R_N\left(\frac
    {\vartheta}{\sqrt{T}} \right),
  \] where $ k=(k_1,\dots, k_d) \in \Z_{\geq 0}^d $ is a multi-index, $
  |k| = k_1 + \cdots + k_d $, and the last term satisfies
  \begin{equation}
    \label{eq:remainder_sf} \left \lvert R_N\left(\frac{\vartheta}{\sqrt
    {T}} \right) \right\rvert \leq C_{a,v} \frac{\vartheta_1^{k_1}
    \cdots \vartheta_d^{k_d}}{(N+1)!  \, T^{\frac{N+1}{2}}},
  \end{equation}
  for a constant $ C_{a,v} $ depending on $ b $, and hence on $ a $ and $
  v $ only. Plugging this expansion into~\eqref{eq:to_estimate_sf}, and
  recalling $ b(0)=a(0) $, we get
  \begin{equation*}
    \begin{split}
      &\frac{1}{\sqrt{\det H}} \int_{\|\vartheta\|\leq \delta} e^{-\frac
      {T}{2}\|\vartheta\|^2} b(\vartheta) \diff \vartheta = \frac{a(0)}{T^
      {\frac{d}{2}}\sqrt{\det H}} \int_{-\infty}^{\infty} e^{-\frac{1}{2}\|\vartheta\|^2}
      \diff \vartheta \\
      & \qquad + \sum_{|k| \leq N} \frac{D^kb(0)}{k!  \, T^{\frac{|k|+d}
      {2}}\sqrt{\det H}} \int_{-\infty}^{\infty} e^{-\frac{1}{2}\|\vartheta\|^2}
      \vartheta^k \diff \vartheta \\
      & \qquad + \frac{1}{T^{\frac{d}{2}}\sqrt{\det H}} \int_{-\infty}^{\infty}
      e^{-\frac{1}{2}\|\vartheta\|^2} R_N\left(\frac{\vartheta}{\sqrt{T}}
      \right) \diff \vartheta + o(e^{-\sqrt{T}}).
    \end{split}
  \end{equation*}
  The conclusion follows from the bound~\eqref{eq:remainder_sf} and the
  fact that the terms with $ |k| $ odd do not appear, since the
  integrals $ \int_{-\infty}^{\infty} e^{-\frac{1}{2}\|\vartheta\|^2}
  \vartheta^k \diff \vartheta $ vanish for $ |k| $ odd.
\end{proof}

%%%%%%%%%%

\subsection{Proof of Theorem~\ref{thm:main1}}

Let us fix $ f_1, f_2 \in \rapid $. By Lemma~\ref{lem:int_of_corrs} and Lemma~\ref{lem:Omega_f}, we can write

\[
  \langle f_1 \circ h_t, f_2 \rangle_{L^2(\Mc)} = \int_{\Galhat} \langle
  \uot. \Omega_\omega(f_1) , \Omega_\omega(f_2) \rangle_{H_\omega} \diff
  \omega,
\] where $ \Omega_\omega(f_j) $ are smooth families of vectors in $ W^r(H_\omega)
$ for every $r \geq 0$.
We now show that the main contribution in the integral above comes from 
the projection on the subspaces $H_{\omega, \lambda_0(\omega)}$ corresponding 
to the smallest positive eigenvalue $\lambda_0(\omega)$ of  the twisted Laplacian $ \Delta_{S,\omega}$.

For any $ \omega \in \harm $, we have an orthogonal decomposition into
invariant subspaces
\[
  H_\omega = H_{\omega, \lambda_0(\omega)} \oplus H_{\omega, \lambda_0(\omega)}^
  {\perp},
\] where the Casimir operator $ \square_\omega $ acts on $ H_{\omega,
\lambda_0(\omega)} $ as the constant $ \lambda_0(\omega) $, and $ H_{\omega,
\lambda_0(\omega)}^{\perp} $ is the orthogonal complement of $ H_{\omega,
\lambda_0(\omega)} $. According to this decoposition, we write
\begin{equation*}
  \begin{split}
    \Omega_\omega(f_j) = \Omega^{\lambda_0(\omega)}(f_j) + \Omega_\omega^{\perp}
    (f_j), \quad \text{where}& \quad \Omega^{\lambda_0(\omega)}(f_j) \in
    W^3(H_{\omega, \lambda_0(\omega)}),\\
    \text{and}& \qquad \Omega_\omega^{\perp}(f_j) \in W^3(H_{\omega, \lambda_0(\omega)}^
    {\perp}).
  \end{split}
\end{equation*}
By Lemma~\ref{thm:other_eigenvalues}, there exists $\delta >0$ such that the 
spectrum of $\square_\omega$ restricted to $H_{\omega, \lambda_0(\omega)}^{\perp}$ 
does not intersect the interval $[0,\delta)$. Thus, by Theorem~\ref{thm:Ratner_mix},
up to changing $\delta >0$, we deduce
\[
\abs{ \langle \uot.  \Omega_\omega^{\perp}(f_1) , \Omega_\omega^{\perp}(f_2)
    \rangle_{H_\omega} } \leq C \|\Omega_\omega(f_1)\|_{W^3} \, 
\|\Omega_\omega(f_1)\|_{W^3} t^{-\delta},
\]
for some absolute constant $ C >0 $.
This implies that 
\begin{equation*}
  \begin{split}
    &\abs{ \langle f_1 \circ h_t, f_2 \rangle_{L^2(\Mc)} 
-  \int_{\Galhat} \langle \uot.  \Omega^{\lambda_0(\omega)}(f_1)
    , \Omega^{\lambda_0(\omega)} (f_2) \rangle_{H_\omega} } \diff \omega \\
    & \qquad \leq C \|f_1\|_{W^3} \, \|f_2\|_{W^3} t^{-\delta}.
  \end{split}
\end{equation*}

We now focus on the space $ H_{\omega, \lambda_0(\omega)} $ and we show that we 
can reduce to consider a neighbourhood of $0 \in \Galhat$. 
By the
representation theory of compact Abelian groups, $ H_{\omega, \lambda_0(\omega)}
$ admits a orthonormal basis of eigenvectors $ e_j(\omega) $ of $ \Theta_\omega
$, with eigenvalue $ j \in \Z $. Any such element $ e_j(\omega) $ is an
eigenvector of the elliptic selfadjoint operator $ \Delta_\omega =
\square_\omega -2\Theta^2_\omega $, hence, by classical results in
perturbation theory (see, e.g.,
\cite[p.~392]{Kat}), we can choose these bases to depend analytically in
$ \omega $ in a neighborhood $ U \subset \Galhat $ of $ 0 $, 
so that $\omega \mapsto  \Omega^{\lambda_0(\omega)}(f_j)$, for $j=1,2$, 
are analytic families of smooth vectors. Since $ 0=\lambda_0
(0) $ is a global minimum, again by Lemma~\ref{lem:Omega_f}, up to changing the constants $ C $ and $ \delta $, we
obtain
\begin{multline}
  \label{eq:step1}
    \abs{\langle f_1 \circ h_t, f_2 \rangle_{L^2(\Mc)} - \int_{U}
    \langle \uot. \Omega^{\lambda_0(\omega)}(f_1) , \Omega^{\lambda_0(\omega)}
    (f_2) \rangle_{H_\omega} \diff \omega } \\
    \leq C \|f_1\|_{W^3} \, \|f_2\|_{W^3} t^{-\delta}.
\end{multline}
Moreover, we can assume that $ \lambda_0(\omega) < 1/4 $ for all $
\omega \in U $.

Finally, we now show that we can use the result of Section~\ref{sec:strong_ratner}, 
together with the \lq\lq stationary phase\rq\rq\ estimate of Lemma~\ref{thm:stationary_phase}. 
Combining~\eqref{eq:step1} with Corollary~\ref{cor:strong_Ratner} and the 
Cauchy-Schwarz Inequality, we obtain
\begin{equation}
  \label{eq:step2} \abs{\langle f_1 \circ h_t, f_2 \rangle_{L^2(\Mc)} -
  \int_{U} A_\omega(f_1,f_2) t^{-1+\nu_0(\omega)} \diff \omega } \leq C
  \|f_1\|_{W^3} \, \|f_2\|_{W^3} t^{-\delta},
\end{equation}
where $ \nu_0(\omega) = \sqrt{1-4\lambda_0(\omega)} $ and $A_\omega(f_1,f_2)$ 
is defined in \eqref{eq:def_avw} in Corollary~\ref{cor:strong_Ratner}. 
In order to
conclude, we show that we can apply Lemma~\ref{thm:stationary_phase} to the integral above, which can be rewritten
as
\[
  \int_{U} A_\omega(f_1,f_2) t^{-1+\nu_0(\omega)} \diff \omega =\int_{U}
  A_\omega(f_1,f_2) e^{(\log t)(-1+\nu_0(\omega))} \diff \omega.
\] 
\begin{lemma}
The map $\omega \mapsto A_\omega(f_1,f_2)$ is a smooth map from $\harm$ 
to $\C$.
\end{lemma}
\begin{proof}
By Lemma~\ref{lem:Omega_f} and the previous discussion, 
$v(\omega):=\Omega^{\lambda_0(\omega)}(f_1)$ and 
$w(\omega):=\Omega^{\lambda_0(\omega)}(f_2)$ are smooth 
families of vectors in $W^r(H_\omega)$ for every $r \geq 0$.
By the definition \eqref{eq:def_avw} of $A_\omega(f_1,f_2)$ from 
Corollary~\ref{cor:strong_Ratner}, we need to verify that the derivatives of 
the function 
\[
r^{-1-\nu(\omega)} F_{v(\omega),w(\omega)}(r)
\]
are integrable over $[1,\infty)$. This follows from Lemma~\ref{lem:uot_smooth} 
and the bound 
\[
\left\lvert \frac{\diff^n}{(\diff \omega)^n} r^{-1-\nu(\omega)} \right\rvert 
\leq C_n(\omega)  (1+\log r)^n r^{-1-\nu(\omega)},
\]
for some constant $C_n(\omega)$ independent of $r\geq 1$.
\end{proof}
Finally, by Proposition~\ref{thm:deriv_eigenvalues}, we have that
\[
  D(-1+\nu_0(\omega))(0) = - D(\sqrt{1-4\lambda_0(\omega)})(0) = -2 D\lambda_0
  (0) = 0,
\] and also that
\[
  H:=D^2(1-\nu_0(\omega)(0) = -D^2(\sqrt{1-4\lambda_0(\omega)}) = 2D^2\lambda_0
  (0)
\] is positive definite. 
Therefore, by~\eqref{eq:step2} and by Lemma~\ref{thm:stationary_phase}, there exists constants $ (c_j)_{j \in \N} $
such that for every $ N \geq 1 $ we can write
\[
  \begin{split}
    & \langle f_1 \circ h_t, f_2 \rangle_{L^2(\Mc)} \\
    & \qquad = \frac{(2\pi)^{\frac{d}{2}}}{ (\log t)^{\frac{d}{2}} \sqrt
    {\det H} } A_0(f_1,f_2) + \sum_{j=1}^N \frac{c_j}{ (\log t)^{j+\frac
    {d}{2}}} + o\left( (\log t)^{-N-\frac{d}{2}}\right) \\
    & \qquad =\left(\frac{g-1}{2}\right)^{\frac{d}{2}} \sigma(\Gammac)
    \frac{A_0(f_1,f_2)}{(\log t)^{\frac{d}{2}}} + \sum_{j=1}^N \frac{c_j}
    { (\log t)^{j+\frac{d}{2}}}+ o\left( (\log t)^{-N-\frac{d}{2}}\right),
  \end{split}
\] where $ \sigma(\Gammac) $ is the determinant of the matrix associated
to the bilinear form $ \langle \eta, \zeta\rangle $ on the space $ \harm
$ of harmonic one-forms.

Finally,
\begin{equation*}
  \begin{split}
    A_0(f_1,f_2) &= \lim_{t \to \infty} \langle \Omega_0(f_1 \circ h_t),
    \Omega_0(f_2) \rangle_{H_0} \\
    &= \int_{M}\Omega_0(f_1) \diff
    \vol \cdot \int_{M} \overline{ \Omega_0(f_2) }\diff
    \vol \\
    &= \int_{M} \pi_\one (f_1) \diff
    \vol \cdot \int_{M} \overline{ \pi_\one(f_2) } \diff
    \vol =\int_{M_0} f_1 \diff
    \vol \cdot \int_{M_0} \overline{ f_2 }\diff
    \vol,
  \end{split}
\end{equation*}
which completes the proof of Theorem~\ref{thm:main1}.

%%%%%%%%%%

\section{Proof of Theorem~\ref{thm:main2}}%
\label{sec:proof_thm_B}

In this section we prove Theorem~\ref{thm:main2} using the theory we developed in Section~\ref{sec:spectrum-an-abelian}.

Let us denote by $ \Gal_k $ the Galois group $ \Gamma / \Gamma_k $ of
the cover $ p_k \colon M_k \to M $.  The surjective map $ \psi \colon
\Gamma \to \Z^d $, as in the assumptions, induces an isomorphism of
finite Abelian groups
\[
  \Gal_k \simeq \Z / (N_1^{(k)}\Z) \times \cdots \times \Z / (N_d^{(k)}\Z).
\] By Lemma~\ref{lem:Abelian_covers:1}, for any $ \chi \in \Galhat_k $, there exist $
\omega \in \harm $ such that
\[
  \chi([\gamma]) = \exp\left( 2 \pi \imath \int_{x}^{[\gamma].x} j_k^{\ast}
  \, p_k^{\ast} \, \omega\right),
\] where $ j_k \colon M_k= T^1(S_k) \to S_k $ is the canonical
projection. More precisely, the same argument as in Lemma~\ref{lem:Abelian_covers:1} shows that $ \Galhat_k $ coincides with the
exponential of $ \harm(k)/\harm(\Z) $, where $ \harm(k) $ is the $ \Z $-module
of all harmonic forms $ \omega $ such that
\[
  \int_{x}^{[\gamma].x} j_k^{\ast} \, p_k^{\ast} \, \omega \in \Z \qquad
  \text{ for all } [\gamma] \in \Gal_k.
\] Notice that
\[
  \harm(k)/\harm(\Z) \simeq (1/N_1^{(k)})\Z/\Z \times \cdots \times (1/N_d^
  {(k)})\Z/\Z.
\] As in \S~\ref{sec:line_bundles}, let us denote by $ L^2(M_\chi) $ the space of $
L^2 $-sections of the line bundle with holonomy $ \chi $.
\begin{lemma}
  We have
  \[
    L^2(M_k) = \bigoplus_{\omega \in \harm(k)/\harm(\Z)} L^2(\Mchio),
  \] in particular
  \[
    \Spec(\square_j) \cap \R_{\geq 0}= \bigcup \{\lambda_n(\omega) :  n
    \in \N,\ %
    \omega \in \harm(k)/\harm(\Z) \}.
  \]
\end{lemma}
\begin{proof}
  Let $ N_j := \abs{\Gal_k} = N_1^{(k)} \cdots N_d^{(k)} $. It is not
  hard to see that the map
  \[
    \begin{split}
      \Phi \colon \bigoplus_{\omega \in \harm(k)/\harm(\Z)} L^2(\Mchio)
      & \to L^2(M_k) \\
      \left( s_\omega \right)_{\omega \in \harm(k)/\harm(\Z)} &\mapsto
      \frac{1}{N_j} \sum_{\omega \in \harm(k)/\harm(\Z)} s_\omega
    \end{split}
  \] is a well-defined isometry whose image contains the dense subspace
  of smooth functions $ \mathscr{C}^\infty(M_j) $, from which the
  conclusion follows.
\end{proof}

By Lemma~\ref{thm:other_eigenvalues}, there exists $ \delta >0 $ such that $
\lambda_0(\omega) $ is analytic for all $ \omega $ such that $ |\lambda_0
(\omega)|<\delta $, and $ \lambda_n(\omega) > \delta $ for all $ n \geq
1 $ and all $ \omega \in \harm(k) / \harm(\Z) $. In particular,
\[
  \Spec(\square_j) \cap [0, \delta] = \{ \lambda_0(\omega) :  \omega \in
  \harm(k) / \harm(\Z) \} \cap [0, \delta].
\]

Let $ f $ be a continuous function compactly supported in $ [0,\delta) $.
We have
\[
  \frac{1}{| \Gamma / \Gamma_k |} \sum_{\lambda \in \Spec(\square_k)
  \cap [0, \delta]} f(\lambda) = \frac{1}{| \Gal_k |} \sum_{\omega \in
  \harm(k) / \harm(\Z)} f(\lambda_0(\omega)),
\] so that, by equidistribution of the points $ \omega \in \harm(k) /
\harm(\Z) $ in the torus $ \harm / \harm(\Z) $,
\[
  \begin{split}
    \lim_{k \to \infty} \frac{1}{| \Gamma / \Gamma_k |} \sum_{\lambda
    \in \Spec(\square_k) \cap [0, \delta]} f(\lambda) &= \int_{\harm /
    \harm(\Z)} f(\lambda_0(\omega)) \diff \omega \\
    &= \int_{\lambda_0^{-1}([0,\delta])} f(\lambda_0(\omega )) \diff
    \omega.
  \end{split}
\] Let us rewrite the integral on the right.  By the Morse Lemma (as in
the proof of Lemma~\ref{thm:stationary_phase}), up to choosing a smaller $ \delta $, we can
find a change of variable $ \omega = \Psi(\mathbf{y}) $, with $ \mathbf{y}
\in \R^d $, such that $ \lambda_0( \Psi(\mathbf{y})) = y_1^2 + \cdots +
y_d^2 $.  In particular
\begin{equation*}
  \begin{split}
    \int_{\lambda_0^{-1}([0,\delta])} f(\lambda_0(\omega)) \diff \omega
    &= \int_{\Psi^{-1} (\lambda_0^{-1}([0,\delta]))} f \left(y_1^2 +
    \cdots + y_d^2 \right) \Psi'(\mathbf{y}) \diff \mathbf{y} \\
    &= \int_{|\mathbf{y}|^2< \delta} f \left(y_1^2 + \cdots + y_d^2
    \right) \widehat{\Psi}(\mathbf{y}) \diff \mathbf{y},
  \end{split}
\end{equation*}
where
\[
  \widehat{\Psi}(\mathbf{y}) = \Psi'(\mathbf{y}) \one_{\Psi^{-1} (\lambda_0^
  {-1}([0,\delta]))}(\mathbf{y}).
\] Passing to polar coordinates, we can write
\[
  \int_{|\mathbf{y}|^2< \delta} f \left(y_1^2 + \cdots + y_d^2 \right)
  \widehat{\Psi}(\mathbf{y}) \diff \mathbf{y} = \int_0^{\sqrt{\delta}} f
  (r^2) r^{d-1} {\zeta}(r) \diff r,
\] for a function $ \zeta\in L^{\infty}([0, \delta]) $.  Finally,
setting $ x=r^2 $, we conclude
\[
  \lim_{k \to \infty} \frac{1}{| \Gamma / \Gamma_k |} \sum_{\lambda \in
  \Spec(\square_k) \cap [0, \delta]} f(\lambda) = \int_0^{\delta} f(x) x^
  {\frac{d}{2}-1} \frac{\zeta(\sqrt{x})}{2} \diff x.
\] Thus, the proof of Theorem~\ref{thm:main2} is complete.

\subsection*{Acknowledgements}
We would like to thank the referees for their comments which helped 
improving the presentation of the paper.

%%%%%%%%%%

\end{document}